\documentclass[10pt,a4paper]{amsart}
\usepackage[leqno]{amsmath}
\usepackage{amsthm}
\usepackage{amsfonts}
\usepackage{amssymb}
\usepackage{eucal}
\usepackage[all]{xy}
\CompileMatrices

\usepackage{mathrsfs}

\usepackage{amsmath,amssymb,mathrsfs,xy,amsthm,bm,url}
\input{xy}
\xyoption{all}

\usepackage{xspace}

 \usepackage[ansinew]{inputenc}
 \usepackage[T1]{fontenc}
 \usepackage{hyperref}

\usepackage{hyperref}

\theoremstyle{definition}
  \newtheorem{definition}[subsection]{Definition}
  
  \newtheorem{definition-proposition}[subsection]{Definition-Proposition}
  \newtheorem{example}[subsection]{Example}
  \newtheorem{remark}[subsection]{Remark}

\theoremstyle{theorem}
  \newtheorem{theorem}[subsection]{Theorem}
  \newtheorem{proposition}[subsection]{Proposition}
  \newtheorem*{proposition*}{Proposition}
  \newtheorem{lemma}[subsection]{Lemma}
  \newtheorem{corollary}[subsection]{Corollary}
  \newtheorem{conjecture}[subsection]{Conjecture}

  \newtheorem{assumption}[subsection]{Assumption}

\newcommand{\maxx}{{\mathrm{max}}}
\newcommand{\rhorig}{{$\rho$-rigid}\xspace}

\newcommand{\adele}{\mathbb{A}_\mathrm{f}}
\newcommand{\Abb}{\mathbb{A}}
\newcommand{\Cbb}{\mathbb{C}}
\newcommand{\Gbb}{\mathbb{G}}
\newcommand{\Nbb}{\mathbb{N}}
\newcommand{\Qbb}{\mathbb{Q}}
\newcommand{\Rbb}{\mathbb{R}}
\newcommand{\Sbb}{\mathbb{S}}
\newcommand{\Zbb}{\mathbb{Z}}
\newcommand{\Zbhat}{\hat{\mathbb{Z}}}

\newcommand{\Cbf}{\mathbf{C}}
\newcommand{\Hbf}{\mathbf{H}}
\newcommand{\Gbf}{\mathbf{G}}
\newcommand{\ibf}{\mathbf{i}}

\newcommand{\Mbf}{\mathbf{M}}
\newcommand{\Nbf}{\mathbf{N}}
\newcommand{\Pbf}{\mathbf{P}}

\newcommand{\Tbf}{\mathbf{T}}
\newcommand{\Ubf}{\mathbf{U}}
\newcommand{\Vbf}{\mathbf{V}}
\newcommand{\Wbf}{\mathbf{W}}

\newcommand{\Zbf}{\mathbf{Z}}
\newcommand{\GLbf}{\mathbf{GL}}

\newcommand{\gfrak}{\mathfrak{g}}

\newcommand{\pfrak}{\mathfrak{p}}
\newcommand{\Xfrak}{\mathfrak{X}}
\newcommand{\Yfrak}{\mathfrak{Y}}

\newcommand{\arm}{\mathrm{a}}
\newcommand{\mrm}{\mathrm{m}}

\newcommand{\Hscr}{\mathscr{H}}

\newcommand{\Pscr}{\mathscr{P}}

\newcommand{\Acal}{\mathcal{A}}
\newcommand{\Ccal}{\mathcal{C}}
\newcommand{\Dcal}{\mathcal{D}}
\newcommand{\Fcal}{\mathcal{F}}
\newcommand{\Pcal}{\mathcal{P}}

\newcommand{\Ucal}{\mathcal{U}}
\newcommand{\Wcal}{\mathcal{W}}
\newcommand{\Xcal}{\mathcal{X}}
\newcommand{\Ycal}{\mathcal{Y}}

\newcommand{\Pcalhat}{\hat{\mathcal{P}}}

\newcommand{\ad}{\mathrm{ad}}
\newcommand{\Ad}{\mathrm{Ad}}

\newcommand{\der}{\mathrm{der}}

\newcommand{\ra}{\rightarrow}
\newcommand{\lra}{\longrightarrow}

\newcommand{\mono}{\hookrightarrow}
\newcommand{\isom}{\cong}

\newcommand{\Hom}{\mathrm{Hom}}
\newcommand{\Group}{\mathrm{Group}}
\newcommand{\GSp}{\mathrm{GSp}}
\newcommand{\Sp}{\mathrm{Sp}}
\newcommand{\Ker}{\mathrm{Ker}}
\newcommand{\Lie}{\mathrm{Lie}}

\newcommand{\Res}{\mathrm{Res}}
\newcommand{\Supp}{\mathrm{Supp}}
\newcommand{\Int}{\mathrm{Int}}
\newcommand{\Stab}{\mathrm{Stab}}

\newcommand{\bsh}{\backslash}
\newcommand{\inv}{{-1}}
\newcommand{\ot}{\overset}

\newcommand{\wrt}{{with\ respect\ to}\xspace}

\newcommand{\cosg}{{compact\ open\ subgroup}\xspace}
\newcommand{\cosgs}{{compact\ open\ subgroups}\xspace}

\title{On Special Subvarieties of Kuga Varieties}

\author{Ke Chen} 

\address{Institut f\"ur Mathematik, Johannes Gutenberg Universit\"at Mainz, 55099 Mainz, Deutchland }
\email{chenk@uni-mainz.de} 


\begin{document}

\begin{abstract}  In this paper we prove the equidistribution of certain families of special subvarieties in Kuga varieties, which is a special case of the general Andr\'e-Oort conjecture formulated for mixed Shimura varieties. Our approach is parallel to the pure case treated in the works of L. Clozel, E. Ullmo, and A. Yafaev, which uses tools from ergodic theory.

\end{abstract}

\maketitle

\setcounter{tocdepth}{2}
\tableofcontents

\setlength{\baselineskip}{18pt}

\section*{Introduction}

This paper is an attempt to understand certain aspects of the following conjecture:
\begin{conjecture}(Andr\'e-Oort-Pink, cf.\cite{andre-note},\cite{pink-combination}) Let $M$ be a mixed Shimura variety, and let $Z\subset M$ be an arbitrary closed subvariety. Set $\Sigma(Z)$ to be the set of  special subvarieties of $M$ that are contained in $Z$, and $\Sigma_\maxx(Z)$ the subset of maximal special subvarieties in $\Sigma(Z)$ (for the inclusion order). Then $\Sigma_\maxx(Z)$ is  finite.
\end{conjecture}

In the literature the conjecture is also reformulated in terms of Zariski closure, such as

(\emph{Form-1}) let $(M_n) $ be a sequence of special subvarieties in $M$, then the Zariski closure of $\bigcup_nM_n$ is a finite union of special subvarieties.

and

(\emph{Form-2}) if $(M_n) $ is a sequence of special subvarieties of $M$, strict in the sense that for any $M'\subsetneq M$ special subvariety we have $M_n\nsubseteq M'$ for $n$ large enough, then $\bigcup_nM_n$ is Zariski dense in $M$.
\bigskip

B.Klingler, E.Ullmo, and A.Yafaev have proved the conjecture for pure Shimura varieties assuming the Generalized Riemann Hypothesis, cf. \cite{clozel-ullmo}, \cite{ullmo-yafaev}, \cite{klingler-yafaev}; see also \cite{noot-bourbaki} and \cite{yafaev-bordeaux} for surveys of their works. Their proof mainly consists of two ingredients:

(1) Equidistribution of $\Cbf$-special subvarieties, proved in \cite{clozel-ullmo} and \cite{ullmo-yafaev};

(2) Estimation of the intersection degrees of Hecke correspondences and of Galois orbits of special subvarieties cf. \cite{klingler-yafaev}, and lower bound of Galois orbits of special subvarieties, cf. \cite{ullmo-yafaev}, both inspired by previous works of B.Edixhoven and A.Yafaev, cf.\cite{edixhoven-yafaev}. These estimations depend on the Generalized Riemann Hypothesis. 

The mixed case of the conjecture has aroused interests since \cite{andre-note} (in the form of the universal elliptic curve), and is closely related to the Manin-Mumford conjecture, as suggested in \cite{pink-combination}. The current work is aimed at the equidistribution of $\Cbf$-special subvarieties for Kuga varieties, which is a class of mixed Shimura varieties modeled on the universal family of abelian varieties over pure Shimura varieties. More technical restrictions appear when we transfer the ergodic arguments in \cite{clozel-ullmo}, which leads to our notion of $\rho$-rigidity of $\Cbf$-special subvarieties. With this condition assumed, the input from  the work of S.Mozes and N.Shah \cite{mozes-shah}  can be directly applied in our setting.

The paper is organized as follows:

(1) Section 1 starts with some generalities on mixed Shimura data, and then concentrates on the notions of Kuga data, connected Kuga varieties, and their special subvarieties;   we also define relevant measure-theoretic objects, such as lattice spaces associated to them and  canonical measures on lattice spaces and special subvarieties,   so as to allow further inputs from ergodic theory.

Roughly speaking, a Kuga variety is the universal family of abelian varieties $M$ over a moduli space $S$, with $S$ a suitable pure Shimura variety. Write $f:M\ra S$ for the structure map defining the abelian scheme over $S$, then a special subvariety in $M$ is obtained as a torsion $S'$-subscheme $M''$ of some $M'\ra S'$, where $S'$ is a pure special subvariety of $S$ and $M'\ra S'$ the abelian $S'$-scheme pulled back from $f$; here torsion $S'$-subscheme is analogue to the notion of torsion subvarieties in abelian varieties, namely abelian subvarieties translated by torsion points.

(2) Section 2 defines the notions of $\Cbf$-special sub-objects and $\rho$-rigidity, and  the main results are stated, both for   lattice spaces and for Kuga varieties:

\begin{theorem}\label{main-theorem}
  Let S be either a connected Kuga variety or its associated lattice space, defined by some connected Kuga datum $(\Pbf,Y;Y^+)$. Let $\pi:(\Pbf,Y;Y^+)\ra(\Gbf,X;X^+)$ be the canonical projection onto the pure base, and fix $\Cbf$ a $\Qbb$-torus in $\Gbf$. Denote by $\Pscr'(S)$ the set of \rhorig $\Cbf$-special measures on $S$. Then $\Pscr'(S)$ is compact for the weak topology, and it satisfies the property of "support convergence": if $(\nu_n) $ is a   sequence in $\Pscr'(S)$ converging to $\nu$, then for some $n$ large enough,  $\Supp\mu_n$ is contained in $\Supp\mu$, and $\Supp\mu$ equals the closure of $\bigcup_{n> 0}\Supp\mu_n$ for the archimedean topology.
\end{theorem}
 
From the main theorem one can deduce a special case of the Andr\'e-Oort-Pink conjecture and the equidistribution conjecture for Kuga varieties in the following form:

\begin{corollary}\label{main-corollary} Let  $M$ be a connected Kuga variety, and $\Cbf$ a fixed $\Qbb$-torus in $\Gbf$, following the notions in the theorem above. Take $Z\subset M$ any closed subset for the archimedean topology. Denote by $\Sigma'(Z)$ the set of   \rhorig $\Cbf$-special varieties of $M$ that are contained in $Z$, and by $\Sigma'_\maxx(Z)$ the subset of maximal elements in $\Sigma'(Z)$ for the inclusion order. Then $\Sigma'_\maxx(Z)$ is always finite.
 \end{corollary}

Note that the same statement holds when replacing $M$ by the associated lattice space, since it only involves the archimedean topology.


(3) Section 3 is concerned with the proof for the case of lattice spaces, which relies on a theorem of S.Mozes and N.Shah.

(4) In Section 4 we pass from lattice spaces to Kuga varieties, where we follow the same arguments as used in \cite{clozel-ullmo}, relying on a result of S.Dani and G.Margulis.
 
Note that the current treatment requires little knowledge of the algebraic structure of Kuga varieties. The main theorem is simply an application of some ideas in ergodic theory to a special class of locally symmetric manifold.

\subsection*{Acknowledgement} This paper grows out of author's thesis. He thanks his advisor Prof. Emmanuel Ullmo for guiding him into this subject with great care and patience, without which the paper wouldn't have existed. He thanks Prof. Stefan M\"uller-Stach and Prof. Kang Zuo for their kind help and interest in this work. He also thanks the referee for a careful reading of the manuscript and providing many useful suggestions. The work is partially supported by the project SFB/TR45 ''Periods, moduli spaces, and arithmetic of algebraic varieties''. 
\subsection*{Convention}
Over a base field $k$, a linear $k$-group $\Hbf$ is understood as a smooth affine algebraic $k$-group, and $Z_\Hbf$ resp.
   $Z^\circ\Hbf$ is the center resp. connected center of $\Hbf$.  A $k$-factor of a reductive $k$-group is   a minimal normal non-commutative $k$-subgroup (of dimension $> 0$), i.e. a non-commutative $k$-simple normal $k$-subgroup.

Write $\Sbb$   for the Deligne torus $\Res_{\Cbb/\Rbb}\Gbb_\mrm$, and $\ibf$  a fixed square root of -1 in $\Cbb$. The set of finite adeles is denoted by $\adele$. 

A linear $\Qbb$-group $\Hbf$ is compact if $\Hbf(\Rbb)$ is a compact Lie group. Subscripts and superscripts such as ${}^\circ$, ${}^+$, ${}_+$ follow the convention of P.Delinge in \cite{deligne-pspm}. 

For a real or complex manifold, its archimedean topology is the one locally deduced from the archimedean metric on $\Rbb^n$ or $\Cbb^m$.

\section{Preliminaries}

For $\Hbf$ a linear $\Qbb$-group, we put $\Xfrak(\Hbf):=\Hom_{\Group/\Rbb}(\Sbb,\Hbf_\Rbb)$ resp. $\Yfrak(\Pbf):=\Hom_{\Group/\Cbb}(\Sbb_\Cbb,\Hbf_\Cbb)$, on which $\Hbf(\Rbb)$ resp. $\Hbf(\Cbb)$ acts from the left by conjugation.

Recall the definition of general mixed Shimura data in \cite{pink-thesis} 2.1:

\begin{definition} \label{mixed-shimura-datum}

(1) A \emph{mixed Shimura datum} is a triple $(\Pbf,\Ubf,Y\ot{h}\ra\Yfrak(\Pbf))$, where $\Pbf$ is a connected linear $\Qbb$-group, $\Ubf$ a connected unipotent normal $\Qbb$-subgroup of $\Pbf$, $Y$ a homogeneous space under $\Pbf(\Rbb)\Ubf(\Cbb)$, $h$ a $\Pbf(\Rbb)\Ubf(\Cbb)$-equivariant map with finite fibers, such that for any $y\in h(Y)$:

(MS-1) the composition $\Sbb_\Cbb\ot{y}\lra\Pbf_\Cbb\lra(\Pbf/\Ubf)_\Cbb$ is defined over $\Rbb$;

(MS-2) the composition $\Sbb_\Cbb\ot{y}\lra\Pbf_\Cbb\ot{\Ad}\lra\GLbf_\Cbb(\pfrak_\Cbb)$ defines a rational mixed Hodge structure on $\pfrak=\Lie\Pbf$ of type $\{(-1,-1),(-1,0),(0,-1),(-1,1),(0,0),(1,-1)\}$, with rational weight filtration $W_{-3}=0$, $W_{-2}=\Lie\Ubf$, $W_{-1}=\Lie\Wbf$, and $W_0=\pfrak$; here $\Wbf$ is the unipotent radical of $\Pbf$, namely the maximal connected normal unipotent $\Qbb$-subgroup of $\Pbf$;

(MS-3) The conjugation $\Int(y(\ibf))$ induces a Cartan involution on $(\Pbf/\Wbf)^\ad$, and $(\Pbf/\Wbf)^\ad$ admits no compact $\Qbb$-factors;

(MS-4) Take a rational Levi decomposition $\Pbf=\Wbf\rtimes\Gbf$, and write $\rho$ for the conjugate action of $\Gbf$ on $\Wbf$, then the connected center $\Cbf_\Gbf$ of $\Gbf$ acts on $\Wbf$ through a $\Qbb$-torus isogeneous to a product $\Hbf\times\Gbb_\mrm^d$ with $\Hbf$ some compact $\Qbb$-torus;

(MS-5) $(\Pbf,Y)$ is irreducible in the sense of \cite{pink-thesis} 2.13, i.e. if $\Pbf'\subsetneq\Pbf$ is a $\Qbb$-subgroup, then there exists $y\in\Pbf$ such that $y(\Sbb_\Cbb)\nsubseteq\Pbf'_\Cbb$. We will also refer to $\Pbf$ as the Mumford-Tate group of $Y$ or of the mxied Shimura datum.

We usually write $(\Pbf,\Ubf,Y)$ for mixed Shimura data and denote by $h$ the map $Y\ra\Yfrak(\Pbf)$.

(2) A \emph{pure Shimura datum} is a mixed Shimura datum $(\Pbf,\Ubf,Y)$ as above such that $\Pbf$ is reductive. In this case $\Ubf$ is necessarily trivial, and the space $Y=X$ is a disjoint union of connected Hermitian symmetric spaces of non-compact type, which is proved in \cite{deligne-pspm}.

 A toric Shimura datum is a pure Shimura datum $(\Gbf,X)$ such that $\Gbf$ is a $\Qbb$-torus (hence $X$ is a finite set).
\end{definition}

\begin{remark}
(1) Note that the pure Shimura data in the sense of Pink differ slightly from Deligne's definition in \cite{deligne-pspm}: a pure Shimura datum of Pink is a pair $(\Gbf,X\ot{h}\ra\Xfrak(\Gbf))$ where $X$ a homogeneous space under $\Gbf(\Rbb)$, $h:X\ra \Xfrak(\Gbf)$ is $\Gbf(\Rbb)$-equivariant with finite fibers, such that $(\Gbf,h(X))$ is a pure Shimura datum in the sense of Deligne.

For a general mixed Shimura datum $(\Pbf,\Ubf,Y)$, the spaces $Y$ and $h(Y)$ are complex manifolds, and the morphism $h:Y\ra h(Y)(\subset\Yfrak(\Pbf))$ is an isomorphism when restricted to each connected component of $Y$, cf.\cite{pink-thesis} 2.12. Later we will be mainly concerned with connected Kuga varieties, and both of the two definitions give the same result.

(2) In the definition of mixed Shimura data, the condition (MS-4) is imposed mainly to simplify certain constructions in the study of canonical models of mixed Shimura varieties (over their reflex fields) and constructions of variations of rational mixed Hodge structures with rational weights. In the present paper it is not essentially involved, as we only treat mixed Shimura varieties (in fact mainly Kuga varieties) as complex analytic spaces.

\end{remark}

\begin{definition}\label{shimura-datum-morphism}
(1) \emph{Morphisms} between mixed Shimura data are defined in the evident way: they are of the form $(f,f_*):(\Pbf,\Ubf,Y)\ra(\Pbf',\Ubf',Y')$ where $f:\Pbf\ra\Pbf'$ is a homomorphism of $\Qbb$-groups, with $f(\Ubf)\subset\Ubf'$, and $f_*:Y\ra Y'$ a map between complex manifolds, equivariant \wrt $f:\Pbf(\Rbb)\Ubf(\Cbb)\ra\Pbf'(\Rbb)\Ubf'(\Cbb)$, which fits into the commutative diagram below: $$\xymatrix{
Y \ar[d]_h \ar[r]_{f_*} & Y' \ar[d]_h \\ {\Yfrak(\Pbf)} \ar[r]_{f_*} & {\Yfrak(\Pbf')}}$$ where the lower horizontal $f_*$ is the push-forward $y\mapsto f\circ y$.

A \emph{mixed Shimura subdatum} $(\Pbf,\Ubf,Y)$ of $(\Pbf',\Ubf',Y')$ is a pair $(f,f_*)$ as above such that $f$ is an inclusion of $\Qbb$-subgroup and $f_*$ is injective.

(2) Let $(\Pbf,\Ubf,Y)$ be a mixed Shimura datum, and $\Nbf$ a normal $\Qbb$-subgroup of $\Pbf$, with $\pi_\Nbf:\Pbf\ra\Pbf/\Nbf=:\Pbf'$ the quotient by $\Nbf$. Then the \emph{quotient} of $(\Pbf,\Ubf,Y)$ by $\Nbf$ is $(\Pbf',\Ubf'=\pi_\Nbf(\Ubf),Y')$ with $Y'=\Pbf'(\Rbb)\Ubf'(\Cbb)/\pi_\Nbf(\Stab_{\Pbf(\Rbb)\Ubf(\Cbb)}y)$ for some $y\in Y$. It is universal in the sense that if $(f,f_*):(\Pbf,\Ubf,Y)\ra(\Pbf_1,\Ubf_1,Y_1)$ is a morphism of mixed Shimura data such that $\Nbf$ is annihilated under the homomorphism $f:\Pbf\ra\Pbf_1$, then $(f,f^*)$ admits a unique factorization  $(\Pbf,\Ubf,Y)\overset{\pi_\Nbf}\ra(\Pbf',\Ubf',Y')\ra(\Pbf_1,\Ubf_1,Y_1)$.

In particular, for $\Wbf$ the unipotent radical of $\Pbf$, the quotient of $(\Pbf,\Ubf,Y)$ by $\Wbf$ is a pure Shimura datum $(\Gbf,X)$, which is the \emph{maximal pure quotient} of $(\Pbf,\Ubf,Y)$.

(4) Let $(\Pbf_i,\Ubf_i,Y_i)$ be mixed Shimura data ($i=1,2$), then one can form their \emph{product} in the evident way: $(\Pbf_1\times\Pbf_2,\Ubf_1\times\Ubf_2,Y_1\times Y_2)$.

\end{definition}

We present some general facts about mixed Shimura data.

\begin{proposition} \label{properties} Let $(\Pbf,\Ubf,Y)$ be a mixed Shimura datum, with $\Wbf$ the unipotent radical of $\Pbf$, and $\Pbf=\Wbf\rtimes\Gbf$ a fixed rational Levi decomposition of $\Pbf$.

(1) The Lie group $\Pbf(\Rbb)\Ubf(\Cbb)$ acts on the complex manifold $Y$ continuously by automorphisms. For any irreducible finite-dimensional $\Qbb$-representation $(\Mbf,\rho_\Mbf)$ of $\Pbf$ pure of some weight $n$, there exists a $\Pbf$-equivariant bilinear form $\psi_\Mbf:\Mbf\otimes\Mbf\ra\Qbb(-n)$ such that for any $y\in Y$, the mixed Hodge structure $(\Mbf,\rho_\Mbf\circ h(y))$ is polarized by $\pm\psi_\Mbf$, and the constant local system $\Mbf\times Y$ underlies a polarizable variation of mixed Hodge structures.

If $(f,f_*):(\Pbf,\Ubf,Y)\ra(\Pbf',\Ubf',Y')$ is a morphism of mxied Shimura data, then $f_*:Y\ra Y'$ is a holomorphic map  between complex manifolds. It is a closed immersion when $(\Pbf,\Ubf,Y)$ is a subdatum of $(\Pbf',\Ubf',Y')$. 

(2) Let $(\Pbf,\Ubf,Y)$ be a mixed Shimura datum, with maximal pure quotient $(\Gbf,X)$. Then $\Gbf$ equals Mumford-Tate group of $X$, namely for any $\Qbb$-subgroup $\Hbf\subsetneq\Gbf$, there exists $x\in h(X)$ such that $x(\Sbb_\Cbb)\nsubseteq\Hbf_\Cbb$.

\end{proposition}
\begin{proof}

(1): cf.\cite{pink-thesis} 2.4

(2): if $(\Gbf,X)$ does not satisfy (MS-5), then there exists a $\Qbb$-subgroup $\Gbf'\subsetneq\Gbf$ such that $x(\Sbb)\subset\Gbf'_\Rbb$ for all $x\in h(X)$. Take $\Pbf'=\pi_\Wbf^\inv(\Gbf')$ we see that for all $y\in h(Y)$ we have $y(\Sbb_\Cbb)\subset\Pbf'_\Cbb$.
\end{proof}

We summarize some constructions of Pink as the following:.

\begin{lemma}\label{reconstruction}[cf.\cite{pink-thesis} 2.15, 2.16, 2.18. 2.21]

The quotient $\Vbf=\Wbf/\Ubf$ is a vectorial $\Qbb$-group, which we identify with its Lie algebra. For $\rho$ the action of $\Gbf$ on $\Wbf$, we get algebraic representations $\rho_\Vbf:\Gbf\ra\GLbf_\Qbb(\Vbf)$ and $\rho_\Ubf:\Gbf\ra\GLbf_\Qbb(\Ubf)$, such that for any $x\in h(X)$, we have rational Hodge structures $(\Vbf,\rho_\Vbf\circ x)$ of type $\{(-1,0),(0,-1)\}$ and $(\Ubf,\rho_\Ubf\circ x)$ of type $\{(-1,-1)\}$. $\Wbf$ is a central extension of $\Vbf$ by $\Ubf$ (as unipotent $\Qbb$-groups), and the extension is uniquely given by an alternating bilinear map $\psi_\Wbf:\Vbf\times\Vbf\ra\Ubf$, equivariant \wrt the action of $\rho_\Vbf$ and $\rho_\Ubf$. Moreover $\Gbf$ acts on $\Ubf$ through a split $\Qbb$-torus.

 Conversely, if we are given a pure Shimura datum $(\Gbf,X)$, two algebraic $\Qbb$-representations $\rho_\Vbf:\Gbf\ra\GLbf_\Qbb(\Vbf)$, $\rho_\Ubf:\Gbf\ra\GLbf_\Qbb(\Ubf)$  with the connected center of $\Gbf$ acting through $\Qbb$-tori of the form prescribed in (MS-4), plus a $\Gbf$-equivariant alternating bilinear map $\psi:\Vbf\times\Vbf\ra\Ubf$, giving rise to a central extension $\Wbf$ of $\Vbf$ by $\Ubf$, such that for any $x\in h(X)$, we have rational pure structures $(\Vbf,\rho_\Vbf\circ x)$ of type $\{(-1,0),(0,-1)\}$ and $(\Ubf,\rho_\Ubf\circ x)$ of type $\{(-1,-1)\}$, then by putting $\Pbf=\Wbf\rtimes_\rho\Gbf$, $\rho$ standing for the action of $\Gbf$ on $\Wbf$ obtained from $\rho_\Vbf$, $\rho_\Ubf$ and $\psi$,  and $Y=\Ubf(\Cbb)\Wbf(\Rbb)\times X$, on which $\Pbf(\Rbb)\Ubf(\Cbb)$ acts by the formular $$(w,g)(w',x)=(wg(w'),gx),$$ with $g(w')=\rho(g)(w')$, we get a mixed Shiimura datum $(\Pbf,\Ubf,Y)$.
\end{lemma}

\begin{proposition}\label{description-of-subdata}
 Let $(\Pbf,\Ubf,Y)$ be a mixed Shimura datum, with maximal pure quotient $(\Gbf,X)$, and $\Pbf=\Wbf\rtimes\Gbf$ a rational Levi decomposition. Then up to conjugation by $\Wbf(\Rbb)\Ubf(\Cbb)$ we can find $y\in Y$ such that $h(Y):\Sbb_\Cbb\ra\Pbf_\Cbb$ is defined over $\Rbb$ and has image in $\Gbf_\Rbb\subset\Pbf_\Rbb$. By putting $Y_0=\Gbf(\Rbb)y$, we get a maximal pure subdatum $(\Gbf,Y_0)\mono(\Pbf,\Ubf,Y)$, which is also a section to the projection $\pi_\Wbf:(\Pbf,\Ubf,Y)\ra(\Gbf,X)$. Maximal pure subdata of $(\Pbf,\Ubf,Y)$ are of the form $(w\Gbf w^\inv,wY_0)$, with $w$ running through $\Wbf(\Qbb)$. All sections to $\pi_\Wbf$ are of this form, referred to as the pure sections of $\pi_\Wbf$.

\end{proposition}
\begin{proof}
 Take $y\in h(Y)$, then the image $y:\Sbb_\Cbb\ra\Pbf_\Cbb$ is a $\Cbb$-torus, hence contained in a maximal reductive $\Cbb$-subgroup of $\Pbf$, which is of the form $w\Gbf_\Cbb w^\inv$. Since $\pi_\Ubf\circ y:\Sbb_\Cbb\ra(\Pbf/\Ubf)_\Cbb$, we can take $w\in\Wbf(\Rbb)\Ubf(\Cbb)$ such that $x:=\Int(w^\inv)\circ y$ has image in $\Gbf_\Cbb$ and descends to $\Sbb\ra\Gbf_\Rbb$. It is then easy to check that the rational Hodge structure $\Ad\circ x:\Sbb\ra\Gbf_\Rbb\ra\GLbf_\Rbb(\Lie\Gbf_\Rbb)$ is of type $\{(-1,1),(0,0),(1,-1)\}$, and that $(\Gbf,Y_0:=\Gbf(\Rbb)x)$ is a pure subdatum of $(\Pbf,\Ubf,Y)$, which is a section to $\pi_\Wbf:(\Pbf,\Ubf,Y)\ra(\Gbf,X)$.

Since every Levi $\Qbb$-subgroup of $\Pbf$ is of the form $w\Gbf w^\inv$ for some $w\in\Wbf(\Qbb)$, we see that every pure section of $\pi_\Wbf$ is of the form $(w\Gbf w^\inv,w Y_0)$.\end{proof}

\begin{definition} \label{kuga-datum}

A \emph{Kuga datum} is a mixed Shimura datum $(\Pbf,\Ubf,Y)$ with $\Ubf=1$, namely the mixed Hodge structure induced by $y\in h(Y)$ does not admit any component of type $(-1,-1)$.

By \ref{reconstruction}, a Kuga datum can be reconstructed from a quadruple $(\Gbf,X;\Vbf,\rho)$ where $(\Gbf,X)$ is a pure Shimura datum, $(\Vbf,\rho)$ is an algebraic representation of $\Gbf$ over $\Qbb$, such that for any $x\in h(X)$, $(\Vbf,\rho\circ x)$ is a rational Hodge structure of type $\{(-1,0),(0,-1)\}$; the resulting Kuga datum is $(\Pbf=\Vbf\rtimes_\rho\Gbf,Y=\Vbf(\Rbb)X)$, which is denoted as $(\Pbf,Y)=\Vbf\rtimes_\rho(\Gbf,X)$ to indicate that it is extended from the pure datum $(\Gbf,X)$ by $(\Vbf,\rho)$, and contains $(\Gbf,X)$ as a pure section. We write $\pi:(\Pbf,Y)\ra(\Gbf,X)$ for the projection modulo $\Vbf$.

Notice that if $(\Pbf,Y)$ is a Kuga datum, then its subdata are Kuga data, because no unipotent $\Qbb$-subgroup of Hodge type $(-1,-1)$ arises.

It is also clear that if $(\Pbf,\Ubf,Y)$ is a mixed Shimura datum, then its quotient by $\Ubf$ is a Kuga datum, which is also its maximal Kuga quotient.

As we will mainly work with connected Kuga varieties, it is convenient to consider \emph{connected Kuga data}, which are of the form $(\Pbf,Y;Y^+)$, with $(\Pbf,Y)$ a Kuga datum and $Y^+$ a connected component of $Y$. The rational Levi decomposition gives rise to the notation $(\Pbf,Y;Y^+)=\Vbf\rtimes_\rho(\Gbf,X;X^+)$, where $(\Gbf,X;X^+)$ is a connected pure Shimura datum. The case for mixed Shimura data is similar, but not needed in what follows.
\end{definition}

\begin{example} \label{siegel-example}
Let $(\Vbf,\psi)$ be a (non-degenerate) symplectic space over $\Qbb$.

(1) Let $\GSp(\Vbf)$ be the $\Qbb$-group of symplectic similitudes of $(\Vbf,\psi)$, and $\Hscr(\Vbf)$ the set of complex structures on $\Vbf_\Rbb$ that are polarized by $\pm\psi$. Then $\GSp(\Vbf)(\Rbb)$ acts on $\Hscr(\Vbf)$ transitively, and $(\GSp(\Vbf),\Hscr(\Vbf))$ is a pure Shimura datum (in the sense of Deligne).  

(2) Using the standard representation $\rho:\GSp(\Vbf)\ra\GLbf(\Vbf)$, we define $(\Pcal(\Vbf),\Ycal(\Vbf)):=\Vbf\rtimes_\rho(\GSp(\Vbf),\Hscr(\Vbf))$, then we get the Kuga datum associated to $(\GSp(\Vbf),\Hscr(\Vbf);\Vbf,\rho)$.

(3) Let $\Wcal(\Vbf)$   be the central extension of $\Vbf$ by $\Gbb_\arm$ through the alternating bilinear form $\psi:\Vbf\times\Vbf\ra\Gbb_\arm$, and let $\GSp(\Vbf)$ acts on $\Gbb_\arm$ through the scalar character $\lambda:\GSp(\Vbf)\ra\Gbb_\mrm$ (i.e. $\psi(gv,gv')=\lambda(g)\psi(v,v')$ for $g\in\GSp(\Vbf)$ and $v,v'\in\Vbf$), we get a mixed Shimura datum  $(\Pcalhat(\Vbf),\Ucal(\Vbf),\hat{\Ycal}(\Vbf))$ where $\Pcalhat:=\Wcal(\Vbf)\rtimes\GSp(\Vbf)$, $\Ucal(\Vbf)$ is the center $\Gbb_\arm$ of $\Wcal(\Vbf)$, and $\hat{\Ycal} $ is the $\Wcal(\Vbf)(\Rbb)\Ucal(\Vbf)(\Cbb)$-orbit of $\Hscr(\Vbf)$ in $\Yfrak(\Pcalhat(\Vbf))$ (through $\Hscr(\Vbf)\subset\Xfrak(\GSp(\Vbf))\subset\Yfrak(\Pcalhat(\Vbf))$).

\end{example}


In the rest of this section we collect material about Kuga data, Kuga varieties, and special subvarieties.
\bigskip

\begin{lemma} \label{group-law} Let $(\Pbf,Y)=\Vbf\rtimes_\rho(\Gbf,X)$ be a Kuga datum, with the group law on $\Pbf=\Vbf\rtimes_\rho\Gbf$ written as $$(v,g)(v',g')=(v+g(v'),gg')\ \ \ v,v'\in\Vbf, g,g'\in\Gbf, g(v')=\rho(g)(v'),$$ then

(1) the representation $\rho$ admits no trivial subrepresentations over $\Rbb$ (of dimension $> 0$);

(2) the derived group $\Pbf^\der$ of $\Pbf$ equals $\Vbf\rtimes_\rho\Gbf^\der$.

\end{lemma}

\begin{proof}
(1) If $\Vbf_\Rbb$ admits a trivial subrepresentation $\Vbf'$ over $\Rbb$, then take $x\in h(X)$ we see that $\Sbb\ot{x}\ra\Gbf_\Rbb\ra\Pbf\ra\GLbf_\Rbb(\pfrak_\Rbb)$ induces a real Hodge substructure in $\Vbf_\Rbb$ of type $(0,0)$, which contradicts the fact that $\rho\circ x$ defines a Hodge structure on $\Vbf$ of type $\{(-1,0),(0,-1)\}$.

(2) By the group law formula, for $v\in\Vbf$ and $g\in\Gbf$, $$vgv^\inv g^\inv=(v,1)(0,g)(-v,1)(0,g^\inv)=(v-g(v),1).$$ Since $\Vbf$ does not admit trivial subrepresentation over $\Qbb$, the set $\{v-g(v):v\in\Vbf,g\in\Gbf\}$   generates $\Vbf$. Thus $\Vbf$ (identified as the unipotent radical of $\Pbf$) and $\Gbf^\der$ are both contained in $\Pbf^\der$. It is clear that $\Vbf$ is stabilized by $\Gbf^\der$ under $\rho$, and that $\Pbf/(\Vbf\rtimes_\rho\Gbf^\der)\isom\Gbf/\Gbf^\der$ is commutative, and thus the required equality follows.\end{proof}

We can recover Kuga subdata in the following way:
\begin{lemma}
  \label{subdatum-recovery} Let $(\Pbf,Y)$ be a Kuga datum of the form $\Vbf\rtimes_\rho(\Gbf,X)$. Then any Kuga subdatum of  $(\Pbf,Y)$ can be described as $(\Pbf_1=\Vbf_1\rtimes_\rho(v\Gbf_1v^\inv),Y_1=(\Vbf_1(\Rbb)+v)\rtimes X_1)$, where $(\Gbf_1,X_1)$ is some pure subdatum of $(\Gbf,X)$, $v$ a vector in $\Vbf(\Qbb)$, $\Vbf_1\subset\Vbf$ a subrepresentation of the restriction $\rho_1=\rho_{|\Gbf_1}:\Gbf_1\ra\GLbf_\Qbb(\Vbf)$, and $(\Vbf_1(\Rbb)+v)\rtimes X_1$ stands for the orbit of $X_1$ in $Y$ under the subset $\Vbf_1(\Rbb)+v\subset\Vbf(\Rbb)$; $v$ is unique up to translation by $\Vbf_1(\Rbb)$. Here it is also required that the connected center of $\Gbf_1$ acts on $\Vbf_1$ through a $\Qbb$-torus satisfying the condition \ref{mixed-shimura-datum}(1)(K-4). We thus also write $(\Pbf_1,Y_1)=\Vbf_1\rtimes_\rho(v\Gbf_1v^\inv,v\rtimes X_1)$, where $v\rtimes X_1:=\Int(v)(X_1)\subset\Xfrak(\Pbf)$, $\Int(v)$ being the conjugation by $v$.
\end{lemma}

\begin{proof}
  Take $(\Gbf_1,X_1)$ as the image of $(\Pbf_1,Y_1)$ under $\pi_*:(\Pbf,Y)\ra(\Gbf,X)$, which is a subdatum by \ref{mixed-shimura-datum}(2). For $x\in h(X_1)$, the Hodge structure on $\gfrak_1=\Lie\Gbf_1$ is induced from the one on $\gfrak=\Lie\Gbf$ given by $x$, hence $(\Gbf_1,X_1)$ is pure itself. In particular $\Gbf_1$ is a reductive $\Qbb$-subgroup of $\Gbf$. The kernel of $\Pbf_1\ra\Gbf$ is contained in $\Ker\pi=\Vbf$, which is unipotent. Thus $\Gbf_1$ is a maximal reductive quotient of $\Pbf_1$, and $(\Gbf_1,X_1)$ is a maximal reductive quotient of $(\Pbf_1,Y_1)$. Take a Levi decomposition of $\Pbf_1$ of the form $\Pbf_1=\Vbf_1\rtimes\Hbf$, with $\Vbf_1$ being the unipotent radical, and $\Hbf$ isomorphic to $\Gbf_1$. Then $\Hbf\subset\Pbf_1\subset\Pbf$ extends to a maximal reductive $\Qbb$-subgroup of $\Pbf$, namely a Levi $\Qbb$-subgroup of the form $v^\inv\Gbf v$ for some $v\in\Vbf(\Qbb)$.

If one has a second choice $u\in\Vbf(\Qbb)$ in place of $v$, namely $$\Vbf_1\rtimes(v\Gbf_1v^\inv)=\Pbf_1=\Vbf_1\rtimes(u\Gbf_1u^\inv),$$ then $u\Gbf_1u^\inv$ and $v\Gbf_1v^\inv$ are both Levi $\Qbb$-subgroups of $\Pbf_1$, and they differ by a $\Vbf_1(\Qbb)$-conjugation, i.e. for some $w\in\Vbf_1(\Qbb)$ we have $wu\Gbf_1u^\inv w^\inv=v\Gbf_1v^\inv$. By the formulas in \ref{group-law}, this means $(w+u-v)\Gbf_1(w+u-v)^\inv=\Gbf_1$, and $w+u-v$ is fixed by $\Gbf_1$. Since $(\Vbf\rtimes\Gbf_1,\Vbf(\Rbb)\rtimes X_1)$ is a Kuga datum itself, the representation of $\Gbf_1$ on $\Vbf$ does not admit trivial subrepresentations (of dimension $> 0$), hence it has no fixed vectors in $\Vbf(\Qbb)$. Therefore $w+u-v=0$, and $v$ is unique up to translation by $\Vbf_1(\Qbb)$.\end{proof}

\begin{definition}
  \label{kuga-variety}
  Let $(\Pbf,Y)$ be a Kuga datum.

  (1) To each \cosg $K\subset\Pbf(\adele)$ one associates an algebraic variety $M_K(\Pbf,Y)$ over $\Cbb$, whose complex points are described by the forumla $$M_K(\Pbf,Y)(\Cbb)=\Pbf(\Qbb)\bsh[Y\times\Pbf(\adele)/K]$$ called the \emph{Kuga variety} at level $K$ defined by $(\Pbf,Y)$. In the definition $\Pbf(\Qbb)$ acts on the product $Y\times\Pbf(\adele)/K$ via the diagonal, and the algbraicity is established in \cite{pink-thesis}, which generalizes the pure case treated in \cite{baily-borel}.

  It also follows from \cite{pink-thesis} 3.2 that if one fixes $\{g\}$ a finite set of representatives of the double quotient $\Pbf(\Qbb)_+\bsh\Pbf(\adele)/K$, then the complex points are equally described as $$M_K(\Pbf,Y)(\Cbb)\isom\coprod_g\Gamma_K(g)\bsh Y^+$$ where $Y^+$ is any fixed connected component of $Y$, and $\Gamma_K(g)=\Pbf(\Qbb)_+\cap gKg^\inv$ is a congruence subgroup stabilizing $Y^+$.

  The canonical projection $\wp_K:Y\times\Pbf(\adele)/K\ra M_K(\Pbf,Y)(\Cbb)$ sending $(y,aK)$ to its class modulo $\Pbf(\Qbb)$ is referred to as the \emph{uniformization map} for $M_K(\Pbf,Y)$. Note that if one fixes $\{g\}$ a set of representatives as above, then the subset $Y^+\times gK$ is mapped exactly onto $\Gamma_K(g)\bsh Y^+$.

  (2) Motivated by (1), a \emph{connected Kuga variety} associated to the connected Kuga datum $(\Pbf,Y;Y^+)$ is a complex analytic space of the form $M=\Gamma\bsh Y^+$, where $Y^+$ is a fixed connected component and $\Gamma\subset\Pbf(\Qbb)_+$ is a congruence subgroup. It follows from \cite{pink-thesis} that such a complex analytic space underlies an algebraic variety over $\Cbb$, and every geometrically connected component of a Kuga variety in (1) is obtained in this way. $M$ is referred to as the connected Kuga variety associated to the quadruple $(\Pbf,Y;Y^+,\Gamma)$.

  Since $\Gamma$ contains $\Gamma\cap\Pbf(\Qbb)^+$ as a subgroup of finite index, in the sequel we often take $\Gamma\subset\Pbf(\Qbb)^+$ for simplicity.

 We also have the \emph{uniformization map} $\wp_\Gamma:Y^+\ra\Gamma\bsh Y^+, y\mapsto\Gamma y$.

  (3) If $M_K(\Pbf,Y)$ is given by the Kuga datum $(\Pbf,Y)=\Vbf\rtimes_\rho(\Gbf,X)$ with $K=K_\Vbf\rtimes K_\Gbf$ for \cosgs $K_\Vbf\subset\Vbf(\adele)$ and $K_\Gbf\subset\Gbf(\adele)$, then the canonical projection $$\pi:M=M_K(\Pbf,Y)\ra M_{K_\Gbf}(\Gbf,X)=S$$ defines an abelian $S$-scheme, whose fibers are abelian varieties with $\Vbf(\Abb)/\Vbf(\Qbb)K_\Vbf$ as the underlying complex tori.

  Similarly, in the connected case, if $\Gamma=\Gamma_\Vbf\rtimes\Gamma_\Gbf$ for congruence subgroups $\Gamma_\Vbf\subset\Vbf(\Qbb)$ and $\Gamma_\Gbf\subset\Gbf(\Qbb)_+$ \wrt $(\Pbf,Y;Y^+)=\Vbf\rtimes_\rho(\Gbf,X;X^+)$, then the projection $\pi:\Gamma\bsh Y^+\ra\Gamma_\Gbf\bsh X^+$ is an abelian scheme, whose fibers are isomorphic to $\Gamma_\Vbf\bsh \Vbf(\Rbb)$ (with complex structure varying along the base points in $\Gamma_\Gbf\bsh X^+$). These constructions are used in \cite{pink-combination} 2.9.

\end{definition}

\begin{definition}
  \label{special-subvariety}[cf. \cite{moonen-model} 6.2]
  (1) Let $M=M_K(\Pbf,Y)_\Cbb$ be a complex Kuga variety. A \emph{special subvariety} of $M$ is an algebraic subvariety in $M$ whose complex points are of the form $\wp_K(Y_1^+\times aK)$ with $a\in\Pbf(\adele)$ and $Y_1^+$ a connected component of some Kuga subdatum $(\Pbf_1,Y_1)\subset(\Pbf,Y)$. The readers are referred to \cite{moonen-model} 6.2 for further description of the special subvarieties in the pure case.

  (2) Let $M=\Gamma\bsh Y^+$ be a connected Kuga variety associated to some quadruple $(\Pbf,Y;Y^+,\Gamma)$. A \emph{special subvariety} of $M$ is a closed subvariety  whose complex points are given as $\wp_\Gamma(Y_1^+)$, where $Y_1^+$ is a connected component of some subdatum $(\Pbf_1,Y_1)$ such that $Y_1^+\subset Y^+$.

\end{definition}

\begin{lemma}
  Let $M'\subset M=M_K(\Pbf,Y)$ be a special subvariety in the sense of \ref{special-subvariety}(1), then it can be realized in the sense of (2). Conversely, up to shrinking the congruence subgroup $\Gamma$, a special subvariety in (2) can be realized in the sense of (1).
 
\end{lemma}

\begin{proof}
  Assume that the complex points of $M'$ are given as $\wp_K(Y'^+\times aK)$ for some $a\in\Pbf(\adele)$ and $Y'^+$ coming from some subdatum $(\Pbf',Y')$. Let $Y^+$ be a connected component of $Y$ containing $Y'^+$, and extend $a$ to a finite set of representatives of the double quotient $\Pbf(\Qbb)_+\bsh \Pbf(\adele)/K$. Then by construction in \label{special-subvariety}(1), $\wp_K(Y'^+\times aK)$ is mapped onto a special subvariety contained in the geometrically connected component $\Gamma_K(a)\bsh Y^+$, and thus equals $\wp_{\Gamma_K(a)}(Y'^+)$.

  Conversely, if $M'\subset M=\Gamma\bsh Y^+$ is a special subvariety $M'=\wp_\Gamma(Y'^+)$ for some connected subdatum $(\Pbf',Y';Y'^+)\subset(\Pbf,Y;Y^+)$, then up to shrinking $\Gamma$, we may assume that $\Gamma=\Pbf(\Qbb)_+\cap K$ for some \cosg $K\subset\Pbf(\adele)$. Then $M'$ is the same as $\wp_K(Y'^+\times K)$ in $M_K(\Pbf,Y)$.\end{proof}


In order to apply measure-theoretic results to  Kuga varieties, we need certain lattice spaces associated to $\Qbb$-linear groups.

\begin{definition}
  \label{type-h}
  A linear $\Qbb$-group $\Pbf$ is said to be \emph{of type $\Hscr$} if it is of the form $\Pbf=\Wbf\rtimes\Hbf$ with $\Wbf$ a unipotent $\Qbb$-group and $\Hbf$ a connected semi-simple $\Qbb$-group without compact $\Qbb$-factors. 

  Given a linear $\Qbb$-group of type $\Hscr$ and $\Gamma\subset\Pbf(\Rbb)^+$ a congruence subgroup, we call $\Omega=\Gamma\bsh \Pbf(\Rbb)^+$  the \emph{lattice space} associated to $(\Pbf,\Gamma)$. Because $\Gamma$ is discrete in $\Pbf(\Rbb)^+$, the quotient $\Omega$ (of classes of translation by $\Gamma$) is a smooth manifold. We also have the \emph{uniformization map} $\wp_\Gamma:\Pbf(\Rbb)^+\ra\Omega$, $g\mapsto\Gamma g$.

  The left Haar measure $\nu_\Pbf$ on $\Pbf(\Rbb)^+$ passes to a Borel measure $\nu_\Omega$ on $\Omega$, referred to as the \emph{canonical measure} on the lattice space. More concretely, one can choose any fundamental domain $\Fcal\subset\Pbf(\Rbb)^+$ \wrt $\Gamma$, and put $\nu_\Omega(A)=\nu_\Pbf(\Fcal\cap \wp_\Gamma^\inv(A))$, for $A\subset\Omega$ measurable.

  The measure $\nu_\Omega$ is of finite volume, and is always normalized so that $\nu_\Omega(\Omega)=1$. In fact we have the following:
\end{definition}

\begin{lemma} \label{harish-chandra}
  (1) Let $\Pbf$ be a linear $\Qbb$-group of type $\Hscr$. Then every congruence subgroup of $\Pbf(\Rbb)^+$ is a lattice in the sense of \cite{harish-chandra-borel}.

  (2) Let $(\Pbf,Y)=\Vbf\rtimes_\rho(\Gbf,X)$ be a Kuga datum with pure section. Then $\Pbf^\der=\Vbf\rtimes\rho\Gbf^\der$ is a linear $\Qbb$-group of type $\Hscr$, and $\Pbf^\der(\Rbb)$ is unimodular, i.e. the left invariant Haar measure is also right invariant.

\end{lemma}

\begin{proof}
  (1) It is clear that $\Pbf$ admits only the trivial character  $\Pbf\ra\Gbb_\mrm$ defined over $\Qbb$: in fact we are immediately reduced to the case where $\Pbf$ is reductive itself, and then it is semi-simple by definition. The claim then follows from \cite{harish-chandra-borel} 9.4.

  (2) We have seen in \ref{group-law} that $\Pbf^\der=\Vbf\rtimes_\rho\Gbf^\der$, hence it is of type $\Hscr$ as $\Gbf^\der$ is connected semi-simple $\Qbb$-group. Since $\Pbf^\der(\Rbb)$ is generated by commutators of $\Pbf(\Rbb)$, the modulus function must be trivial on it, hence it is uni-modular.\end{proof}

\begin{definition}
  \label{kuga-lattice-space}

  Let $M$ be the Kuga variety associated to the quadruple $(\Pbf,Y;Y^+,\Gamma)$.

  (1) The \emph{lattice space} associated to $M$ is the quotient $\Omega=\Gamma^\dag\bsh \Pbf^\der(\Rbb)^+$, where $\Gamma^\dag:=\Gamma\cap\Pbf^\der(\Rbb)^+$. $\Omega$ carries  the \emph{canonical (probability) measure} $\mu_\Omega$. Similar to Kuga varieties, the canonical projection $\Pbf^\der(\Rbb)^+\ra\Omega$, $g\mapsto\Gamma^\dag g$ is again referred to as the \emph{uniformization map}, and denoted as $\wp_\Gamma$ by abuse of notations.

  One also has the following \emph{orbit map} $$\kappa_y:\Omega=\Gamma^\dag\bsh\Pbf^\der(\Rbb)^+\ra M=\Gamma\bsh Y^+,\ \ \Gamma^\dag g\mapsto\Gamma gy$$ for any fixed point $y$ in $Y^+$.

  (2) Let $\Omega=\Gamma^\dag\bsh\Pbf^\der(\Rbb)^+$ be a lattice space as in (1). A \emph{lattice subspace} of $\Omega$ is a closed subset of the form $\Omega_\Hbf=\wp(\Hbf(\Rbb)^+)$ where $\Hbf\subset\Pbf^\der$ is a linear $\Qbb$-subgroup of type $\Hscr$. Here $\Hbf$ is not required to come from some Kuga subdatum.

  Since $\Omega_\Hbf\isom(\Gamma\cap\Hbf)\bsh\Hbf(\Rbb)^+$, $\Omega_\Hbf$ is the lattice space attached to the pair $(\Hbf,\Gamma\cap\Hbf(\Rbb)^+)$, and it also carries a canonical Borelian probability measure. In the sequel this measure is regarded as a measure on $\Omega$ with support equal to $\Omega_\Hbf$, referred to as the \emph{canonical measure} associated to $\Omega_\Hbf$.

  When $\Hbf=\Pbf'^\der$ for some Kuga subdatum $(\Pbf',Y')$, we get the \emph{special lattice subspace} associated to $(\Pbf',Y')$. Of course one may talk about special lattice subspace associated to some connected Kuga subdatum, and but the lattice subspace only depends on the $\Qbb$-subgroup: if there are subdata with a common $\Qbb$-group $(\Pbf',Y')$ and $(\Pbf',Y'')$, then they define the same special lattice subspace.

  (3) Note also that $Y^+$ carries a canonical measure, invariant under $\Pbf(\Rbb)^+$-conjugacy, and passes to a probability measure on $\Gamma\bsh Y^+$ (after normalization). Parallel to the case of lattice space, it can also be described via fundamental domains.
 
  Let $M'$ be a special subvariety of $M=\Gamma\bsh Y^+$ defined by some connected subdatum $(\Pbf',Y';Y'^+)$, then $M'$ is the image of the morphism $i:\Gamma'\bsh Y'^+\ra \Gamma\bsh Y^+$, with $\Gamma'=\Gamma\cap\Pbf'(\Rbb)_+$, and by the same arguments as in \cite{ullmo-yafaev} 2.2, $i$ is generically finite, which pushes forward the canonical measure on $\Gamma'\bsh Y'^+$ to a probability measure $\mu_{M'}$ on $M'$. Again we always view $\mu_{M'}$ as a measure on $M$ with support equal to $M'$, and call it the \emph{canonical measure} associated to $M'$.
\end{definition}

For reader's convenience we sketch the proofs of some elementary facts involved in the constructions of canonical measures on connected Kuga varieties:

\begin{lemma}
  \label{constructing-canonical-measures}

  Let $M=\Gamma\bsh Y^+$ be a connected Kuga variety given by $(\Pbf,Y;Y^+,\Gamma)$. Fix a base point $y\in Y^+$.

  (1) The orbit map $\kappa_y:\Pbf^\der(\Rbb)^+\ra Y^+ \ \ g\mapsto gy$ is surjective with compact fibers. The isotropy subgroup $K_y$ of $y$ in $\Pbf^\der(\Rbb)^+$ is a maximal compact subgroup.

  (2) The left-invariant Haar measure $\nu_\Pbf$ on $\Pbf^\der(\Rbb)^+$ passes to a left invariant measure $\mu_Y=\kappa_{y*}\nu_\Pbf$ on $Y^+$, which is independent of the choice of $y$.

  (3) Similarly, the orbit map $\kappa_y:\Gamma^\dag\bsh \Pbf^\der(\Rbb)^+\ra\Gamma\bsh Y^+\ \ \Gamma^\dag g\mapsto\Gamma gy$ is surjective with compact fibers. The push-forward $\kappa_{y^*}$ sends $\nu_\Omega$ to a canonical probability measure on $M$, independent of the choice of $y$.
\end{lemma}

\begin{proof}
  (1) When $(\Pbf,Y;Y^+)=(\Gbf,X;X^+)$ is pure, it is clear that $\Gbf^\der(\Rbb)^+y=X^+$, because the center of $\Gbf(\Rbb)^+$ acts trivially on $X^+$. In this case $K_y$ is a maximal compact subgroup of $\Gbf^\der(\Rbb)^+$.

  In the Kuga case, fix a pure section $(\Gbf,X;X^+)$ given by some Levi decomposition. If $y$ lies in $X^+$ \wrt the pure section $(\Gbf,X;X^+)\mono(\Pbf,Y;Y^+)$, then weight filtration of the rational mixed Hodge structure on $\Lie\Pbf$ splits as the direct sum of $\Lie\Vbf$ of weight -1 with $\Lie\Gbf$ of weight 0. Let $P_y$ be the isotropy subgroup of $y$ in $\Pbf(\Rbb)$. Then $P_y$ is the centralizer of the torus $y(\Sbb)$, and $(\Lie P_y)_\Cbb$ is the part of Hodge type (0,0) in $\Lie\Pbf_\Cbb$, hence $P_y\subset\Gbf(\Rbb)$, and $K_y=P_y\cap\Pbf^\der(\Rbb)^+\subset\Gbf^\der(\Rbb)$.
  It is clear that $\Pbf^\der(\Rbb)^+y=Y^+$, as $Y^+\ra X^+$ is a $\Vbf(\Rbb)$-torsor with section, and $\Pbf^\der(\Rbb)^+=\Vbf(\Rbb)\rtimes\Gbf^\der(\Rbb)^+$ with $\Gbf^\der(\Rbb)^+y= X^+$.

  For the general case, it suffices to conjugate $y$ into the pure section $X^+$ by some $v\in\Vbf(\Rbb)$.

  (2) For the fixed base point $y\in Y^+$, the orbit map $\kappa_y:\Pbf^\der(\Rbb)^+\ra Y^+\ \ g\mapsto gy$ is surjective, whose fibers are compact subgroups of the form $\kappa_y^\inv(gy)=gK_yg^\inv$, $K_y$ being the isotropy subgroup of $y$. Consequently, $\mu_Y=\kappa_{y*}\nu_P$ is well-defined and $\Pbf^\der(\Rbb)$-invariant: $\mu_Y(A)=\nu_P(\kappa_y^\inv(A))$, and $\mu_Y(gA)=\nu_P(\kappa_y^\inv(gA))=\nu_P(g\kappa_y^\inv(A))=\mu_Y(A)$, for $A\subset Y^+$ measurable.

  It is independent of the choice of $y$: if we shift $y$ to some $hy$ with $h\in\Pbf^\der(\Rbb)^+$, and put $\mu'_Y:=\kappa_{hy*}\nu_P$, then $\mu'_Y(A)=\nu_P(\kappa_{hy}^\inv(A))=\nu_P(\kappa_y^\inv(A)h^\inv)=\mu_Y(A)$ for $A\subset Y^+$ measurable, as $\nu_P$ is left and right invariant by \ref{harish-chandra}(2).

  (3) The action of $\Gamma^\dag$ on $\Pbf^\der(\Rbb)^+$ admits a fundamental domain $\Fcal$, and $\nu_\Omega$ is defined as $A\mapsto \nu_P(\Fcal\cap\wp_\Gamma^\inv(A))$, for $A\subset\Omega$ measurable. The case of $M=\Gamma\bsh Y^+$ is similar by choosing a fundamental domain $\Dcal$. See for example \cite{harish-chandra-borel} Section.6 and \cite{borel-ji} III.2 for details on fundamental domains and reduction theories.

  The independence of $\mu_M$ on $y$ is also clear: if $A$ is a measurable subset of $M$, then $$\mu_M(A)=\mu_Y(\Dcal\cap\wp_\Gamma^\inv(A))=
  \nu_P(\kappa_y^\inv(\Dcal\cap\wp_\Gamma^\inv(A)))$$
  and when the base point is shifted to $gy$, one uses $g\Dcal$ as a fundamental domain to compute $\mu'_M=\kappa_{gy*}\nu_\Omega$: $$\mu'_M(A)=\mu_Y(g\Dcal\cap\wp_\Gamma^\inv(A))=
  \nu_P(\kappa_{gy}^\inv(g\Dcal\cap\wp_\Gamma^\inv(A)))
  =\nu_P(g\kappa_y^\inv(\Dcal)\cap g\kappa_y^\inv(\wp_\Gamma^\inv(A)))$$ which equals $\mu_M(A)$.\end{proof}

The construction of canonical measures is compatible with the following morphisms of connected Kuga varieties.

\begin{lemma}
  \label{measure-compatibility}

  Let $M=\Gamma\bsh Y^+$ and $\Omega=\Gamma^\dag\bsh\Pbf^\der(\Rbb)^+$ be as in \ref{kuga-lattice-space}.

(1) Let $M'\subset M$ be a special subvariety defined by $(\Pbf',Y';Y'^+)$, and take $y\in Y'^+\subset Y^+$. Then we have the commutative diagram $$\xymatrix{
  \Omega'\ar[d]^{\kappa_y}\ar[r]^{i} & \Omega \ar[d]^{\kappa_y}\\
  M'\ar[r]^{i} & M}$$ where the horizontal $i$'s are induced from the inclusion $(\Pbf',Y')\subset(\Pbf,Y)$. In particular, $\kappa_{y*}\nu_{\Omega'}=\mu_{M'}$.

  (2) Let $\pi:(\Pbf,Y;Y^+)\ra(\Gbf,X;X^+)$ be the projection onto the pure base, and take $\Gamma_\Gbf$ to be the image $\pi(\Gamma)$. Then under the projection $\pi:\Omega\ra\Omega_G:=\Gamma_\Gbf^\dag\bsh\Gbf^\der(\Rbb)^+$ and $\pi:M\ra S:=\Gamma_\Gbf\bsh X^+$ we have $\pi_*\nu_\Omega=\nu_{\Omega_\Gbf}$ and $\pi_*\mu_M=\mu_S$.
\end{lemma}

\begin{proof}

(1) is clear. To see (2) for $\pi:\Omega\ra\Omega_\Gbf$, we consider the fundamental domains $\Fcal_G$ for $\Gamma^\dag_\Gbf=\pi(\Gamma^\dag)$ on $\Gbf^\der(\Rbb)^+$, and $\Fcal_V$ for $\Gamma_\Vbf:=\Ker(\Gamma^\dag\ra\Gamma_\Gbf^\dag)=\Gamma\cap\Vbf(\Qbb)$ on $\Vbf(\Rbb)$. Then $\Fcal=\Fcal_\Vbf\rtimes \Fcal_\Gbf:=\{(v,g):v\in\Fcal_\Vbf,g\in\Fcal_\Gbf\}$ is a fundamental domain for $\Gamma$ on $\Pbf^\der(\Rbb)^+$.

$(\Pbf^\der(\Rbb)^+,\nu_\Pbf)$ is isomorphic to the direct product of $(\Vbf(\Rbb),\nu_\Vbf)$ with $(\Gbf^\der(\Rbb)^+,\nu_\Gbf)$ as measurable spaces.  Hence for $A\subset \Omega_\Gbf$ measurable, $\nu_{\Omega_\Gbf}(A)=\nu_\Gbf(\Fcal_\Gbf\cap\wp_{\Gamma_\Gbf}^\inv(A))
=\nu_\Vbf(\Fcal_\Vbf)\times\nu_\Gbf(\Fcal_\Gbf\cap\wp_{\Gamma_\Gbf}^\inv(A))
=\nu_\Pbf(\Fcal_\Vbf\times(\Fcal_\Gbf\cap\wp{\Gamma_\Gbf}^\inv(A)))
=\nu_\Pbf(\Fcal\cap\wp_\Gamma(\pi^\inv(A)))=\nu_\Omega(\pi^\inv(A))$.

The case for $M\ra S$ is similar. \end{proof}

We end this section with the example of Siegel modular varieties and their universal Kuga families, more details for which can be found in \cite{pink-thesis} 2.7, 2.8, 2.24, 2.25, and Kuga's book \cite{kuga-book}. We also mention the case of a CM fiber, which motivates the definition of $\rho$-rigidity in the next section.

\begin{example} \label{example}

Let $(\Vbf,\psi)$ be a (non-degenerate) symplectic space over $\Qbb$. Then we have the pure Shimura datum $(\GSp(\Vbf),\Hscr(\Vbf))$ and the Kuga datum $(\Pcal(\Vbf),\Ycal(\Vbf))=\Vbf\rtimes_\rho(\GSp(\Vbf),\Hscr(\Vbf))$. Take $\Gamma=\Gamma_\Vbf\rtimes\Gamma_\Gbf$ a congruence subgroup of $\Pcal(\Vbf)$ respecting the Levi decomposition, then the associated lattice spaces are $\Gamma_\Gbf^\dag\bsh\Sp(\Vbf)(\Rbb)$ and $\Gamma^\dag\bsh\Vbf(\Rbb)\rtimes\Sp(\Vbf)(\Rbb)$.

Take $(\Tbf,x)\subset(\GSp(\Vbf),\Hscr(\Vbf))$ a toric subdatum. Then $x$ is a single point,   $(\Tbf,x)$ does not contain proper subdatum, and the subdata of $\Vbf\rtimes_\rho(\Tbf,x)$ are of the form $\Vbf'\rtimes_\rho(v\Tbf v^\inv,v\rtimes x)$ for $v\in\Vbf(\Qbb)$ and $\Vbf'\subset\Vbf$ subrepresentation of $\Vbf$ over $\Qbb$ for the action of $\Tbf$. It follows that the lattice subspaces associated to such subdata are of the form $\wp_\Gamma(\Vbf'(\Rbb))$, and the translation by $v\in\Vbf(\Qbb)$ is not reflected at the level of lattice space.

The moduli interpretation of these constructions is well-known. Assume for simplicity that $(\Vbf,\psi)$ is the standard symplectic space with $\psi$ given by the matrix $\small{\begin{bmatrix}0 & -I_g\\ I_g &0\end{bmatrix}}$ with $I_g$ the $g\times g$ identity matrix, $g=\frac{1}{2}\dim_\Qbb\Vbf$. We may identify $\GSp(\Vbf)$ with $\GSp_{2g}$, and $\Hscr(\Vbf)$ with the Siegel double spaces $\Hscr_g$ of genus $g$. Take $K(N)$ the kernel of the reduction modulo $N$: $\GSp_{2g}(\Zbhat)\ra\GSp_{2g}(\Zbb/N)$ with $N\in\Zbb_{>6}$, then the Shimura variety $M_{K(N)}(\GSp_{2g},\Hscr_g)$ is the Siegel  moduli space $\Acal_{g,N}$ parameterizing principally polarized abelian varieties of dimension $g$ with full level-$N$ structure, and the mixed Shimura variety $\Xcal_{g,N}:=M_{\Vbf(\Zbhat)\rtimes K(N)}(\Pcal(\Vbf),\Ycal(\Vbf))$ is the universal family of abelian varieties over $\Acal_{g,N}$ (where $\Vbf(\Zbhat)$ is defined \wrt the integral structure for a suitable  basis of $\Vbf$).

Let $s$ be a special point in $\Acal_{g,N}$ given by the toric subdatum $(\Tbf,x)$, corresponding to some CM abelian variety with Mumford-Tate group equal to $\Tbf$. Then this CM abelian variety is the fiber of $\Xcal_{g,N}\ra\Acal_{g,N}$ at $s$, which can be obtained as a special subvariety associated to $\Vbf\rtimes(\Tbf,x)$. Torsion subvarieties in this CM abelian variety can also be realized as special subvarieties of $\Xcal_{g,N}$.
\end{example}
\bigskip

\section{Rigid $\Cbf$-special objects}

In this section we introduce auxiliary conditions on Kuga subdata, which are aimed to keep the unipotent translation fully reflected in the associated lattice subspaces and thus to avoid the case of CM fibers in \ref{example} as much as possible.

\begin{definition} \label{rho-rigid-c-special}
  Let $(\Pbf,Y;Y^+)=\Vbf\rtimes_\rho(\Gbf,X;X^+)$ be a Kuga datum, and $\pi:(\Pbf,Y;Y^+)\ra(\Gbf,X;X^+)$ the canonical projection onto its maximal pure quotient. For $\Gamma\subset\Pbf(\Rbb)^+$   a  congruence subgroup, we have $M=\Gamma\bsh Y^+$ the Kuga variety with associated lattice space $\Omega=\Gamma^\dag\bsh\Pbf^\der(\Rbb)^+$.

  (i) $\Cbf$-special sub-objects (\wrt $\pi$):
 Fix $\Cbf$ a $\Qbb$-torus in $\Gbf$

     (i-0) A reductive $\Qbb$-subgroup $\Gbf'$ of $\Gbf$ is $\Cbf$-special if $\Cbf$ equals the connected center of $\Gbf'$.

     (i-1) A pure subdatum $(\Gbf',X')$ of $(\Gbf,X)$ is $\Cbf$-special if $\Gbf'$ is $\Cbf$-special. A Kuga subdatum $(\Pbf',Y')$ of $(\Pbf,Y)$ is $\Cbf$-special if the image $\pi(\Pbf')$ is $\Cbf$-special.
     The case of connected Kuga subdatum is similar.

  (i-2) A lattice subspace $\Omega'$ of $\Omega$ is said to be $\Cbf$-special if it is associated to some connected $\Cbf$-special subdatum $(\Pbf',Y';Y'^+)$ of $(\Pbf,Y;Y^+)$, i.e. $\Omega'=\wp_\Gamma(\Pbf'^\der(\Rbb)^+)$.

  A measure on $\Omega$ is $\Cbf$-special if it is the canonical measure of some $\Cbf$-special lattice subspace.

  (i-3) Similarly, $\Cbf$-special subvarieties of $M$ are those defined by $\Cbf$-special subdata, and the same is understood for the notion of $\Cbf$-special measures on $M$.

Note that by \ref{mixed-shimura-datum} (1) (MS-4), $\Cbf$ is required to acts on $\Vbf$, hence on any $\Cbf$-stable $\Qbb$-subspace of $\Vbf$, through some $\Qbb$-torus isogeneous to the product of a compact one with a split one.

  (ii) $\rho$-rigid subobjects:

  (ii-0) A reductive $\Qbb$-subgroup $\Gbf'$ of $\Gbf$ is $\rho$-rigid if the representation (by restriction) $\rho:\Gbf'^\der\ra\GLbf_\Qbb(\Vbf)$ does not admit trivial subrepresentations (of dimension $> 0$). Note that the condition is non-trivial even when we take $\Gbf'=\Gbf$; moreover, if $\Gbf$ admits a $\rho$-rigid $\Qbb$-subgroup $\Gbf'$, then $\Gbf$ is $\rho$-rigid itself, as $\Gbf'^\der\subset\Gbf^\der$.

  (ii-1) A pure subdatum $(\Gbf',X')$ of $(\Gbf,X)$ is $\rho$-rigid if so it is with the $\Qbb$-group $\Gbf'$. A Kuga subdatum $(\Pbf',Y')$ is of $\rho$-rigid if so it is with the reductive $\Qbb$-subgroup $\pi(\Pbf')$. The case of connected Kuga subdata is parallel.

  (ii-2) A lattice subspace $\Omega'$ of $\Omega$ is of non-CM type if it is defined by some $\rho$-rigid Kuga subdatum. A $\rho$-rigid special measure on $\Omega$ is defined as  the canonical measure associated to some $\rho$-rigid lattice subspace.

  (ii-3) Similarly, the notion of $\rho$-rigid special subvarieties in $M$ and $\rho$-rigid special measures  on $M$ are defined as those associated to $\rho$-rigid Kuga subdata.

  (iii) Finally for $S$ being either $\Omega$ or $M$, we denote by $\Pscr'(S)$ the set of $\rho$-rigid $\Cbf$-special measures on $S$, viewed as a subset of $\Pscr(S)$ the set of Borelian probability measures on $S$ endowed with the weak topology.
\end{definition}

\begin{example}

Let $(\Vbf,\psi)$ be the symplectic space as in \ref{example}, then the Kuga datum $(\Pcal(\Vbf),\Ycal(\Vbf))=\Vbf\rtimes_\rho(\GSp(\Vbf),\Hscr(\Vbf))$ is \rhorig, because the action of $\Sp(\Vbf)$ on $\Vbf$ does not admit trivial subrepresentations.

On the other hand, the subdatum $\Vbf\rtimes_\rho(\Tbf,x)$ is not \rhorig, as the derived group of $\Tbf$ is trivial. 
\end{example} 

\begin{lemma}Let $(\Pbf,Y)=\Vbf\rtimes_\rho(\Gbf,X)$ be a Kuga datum, with $\rho:\Gbf\ra\GLbf_\Qbb(\Vbf)$ faithful. If the connected center of $\Gbf$ splits over $\Qbb$, then $(\Pbf,Y)$ is $\rho$-rigid, namely $\Gbf^\der$ acts on $\Vbf$ without trivial subrepresentations (of dimension $> 0$).
\end{lemma}

\begin{proof}
Assume $\Ubf:=\Vbf^{\Gbf^\der}\neq0$. Then $\Ubf$ is a subrepresentation of $\Vbf$ under $\Gbf$. In particular, for any $x\in X$, $\rho\circ x:\Sbb\ra\Gbf_\Rbb\ra\GLbf_\Rbb(\Ubf_\Rbb)$ defines a rational Hodge substructure of $(\Vbf,\rho\circ x)$, hence of Hodge type $\{(-1,0),(0,-1)\}$.

On the other hand, $\rho\circ x:\Sbb^1\mono\Sbb\ra\Gbf_\Rbb\ra\GLbf_\Rbb(\Ubf_\Rbb)$ factors through $\Gbf^\der_\Rbb$ because the connected center of $\Gbf$ already splits over $\Rbb$. The triviality of the action of $\Gbf^\der$ on $\Ubf$ implies that $\Sbb^1\ra\GLbf_\Rbb(\Ubf_\Rbb)$ is also trivial, and thus $(\Ubf,\rho\circ x)$ is of even weight, which contradicts Hodge type $\{(-1,0),(0,-1)\}$ of $(\Ubf,\rho\circ x)$ computed above.
\end{proof}

\begin{remark}
If $(\Pbf,Y)=\Vbf\rtimes_\rho(\Gbf,X)$ is a Kuga datum with $\rho$ faithful, then the connected center of $\Gbf$ has to be isogeneous to a product of the form $\Hbf\times\Gbb_\mrm^d$ for some compact $\Qbb$-torus $\Hbf$, as is required by \ref{mixed-shimura-datum} (MS-4). If the condition (MS-4) is removed, then the lemma above still makes sense if we only require the connected center of $\Gbf$ to split over $\Rbb$, or simply has some totally real field as the splitting field.
\end{remark}


For $(\Pbf',Y')=\Vbf'\rtimes_\rho(v\Gbf'v^\inv,v\rtimes X')$ a subdatum of $(\Pbf,Y)=\Vbf\rtimes_\rho(\Gbf,X)$, we have shown in \ref{subdatum-recovery} that $v$ is unique up to translation by $\Vbf'$. Using similar arguments, we can show that $\rho$-rigid $\Cbf$-special subdata can be recovered from the corresponding special lattice subspaces:

\begin{lemma} \label{injectivity}
  Let $\Omega=\Gamma^\dag\bsh\Pbf^\der(\Rbb)^+$ be the lattice space associated to a Kuga variety $M$ defined by $(\Pbf,Y;Y^+,\Gamma)$, and $\Cbf$ a fixed $\Qbb$-torus in $\Gbf$. Let $\Ccal$  be the set of $\rho$-rigid $\Cbf$-special Mumford-Tate $\Qbb$-subgroups of $\Pbf$, namely   $\Qbb$-subgroups $\Pbf'$ coming from  $\rho$-rigid $\Cbf$-special subdata
 $(\Pbf',Y')$of $(\Pbf,Y)$. Then the following map
  $$\Ccal\ra\{\mathrm{lattice\ subspaces\ of\ }\Omega\}\ \ \Pbf'\mapsto\wp_\Gamma(\Pbf'^\der(\Rbb)^+)$$
  is injective.
\end{lemma}

\begin{proof}
  Assume that two $\rho$-rigid $\Cbf$-special subdata $(\Pbf_i,Y_i)$ ($i=1,2$) of $(\Pbf,Y)$ give the same lattice space $\Omega'$ under the map above. Then $\Pbf_1^\der=\Pbf_2^\der$ holds by computing the tangent space of $\Omega'$ at the specific point $\Gamma^\dag e$, $e$ being the neutral element of $\Pbf^\der(\Rbb)^+$. It remains to extend the equality into $\Pbf_1=\Pbf_2$.

  We take the decomposition $\Pbf_i=\Vbf_i\rtimes(v_i \Gbf_iv_i^\inv)$ in the sense of \ref{subdatum-recovery}, with $v_i\in\Vbf(\Qbb)$, $\Gbf_i=\pi(\Pbf_i)$ $\Cbf$-special $\Qbb$-subgroup of $\Gbf$, and $\Vbf_i$ a subrepresentation of $\Vbf$ under $\Gbf_i$. Then $\Vbf_1=\Vbf_2$ as the common unipotent radical of $\Pbf_i^\der$, and $\Gbf_1^\der=\pi(\Pbf_1^\der)=\Gbf_2^\der$. Hence $\Gbf_1=\Gbf_2$ because they are known to share the common connected center $\Cbf$.

  To obtain the equality $\Pbf_1=\Pbf_2$, it suffices to show that $v_1\Gbf_1v_1^\inv$ and $v_2\Gbf_2v_2^\inv$ are conjugate under $\Vbf_1(\Qbb)$. Note that we already have $\Pbf_1^\der=\Pbf_2^\der$, which is nothing but
  $$\Vbf_1\rtimes(v_1\Gbf_1^\der v_1^\inv)=\Vbf_1\rtimes(v_2\Gbf^\der_1v_2^\inv)$$ Now that $v_1\Gbf^\der_1v_1^\inv$ and $v_2\Gbf^\der_1v_2^\inv$ are both Levi $\Qbb$-subgroups of $\Pbf_1^\der$, they are conjugate under $\Vbf_1(\Qbb)$, hence for some $v\in\Vbf_1(\Qbb)$ we have $$(v+v_1-v_2)\Gbf^\der_1(v+v_1-v_2)^\inv=\Gbf^\der_1$$ By the group law formula in \ref{group-law}, this equality implies that $v+v_1-v_2$ is fixed by the action of $\Gbf_1^\der$ (through $\rho$). $\Gbf_1$ is known to be $\rho$-rigid, therefore $v+v_1-v_2=0$, i.e. $v_1-v_2\in\Vbf_1(\Qbb)$, whence the equality $\Pbf_1=\Pbf_2$.\end{proof}

\begin{remark}
  If the condition of $\rho$-rigidity is not satisfied, one may find different $\Cbf$-special subdata giving rise to the same $\Cbf$-special lattice subspace, see for example the torsion subvarieties in a CM fiber in \ref{example}.

  For a general mixed Shimura datum $(\Pbf,\Ubf,h:Y\ra\Yfrak(\Pbf))$ with maximal pure quotient $(\Gbf,X)$ such that $\Ubf\neq0$, it is showed in \cite{pink-thesis} 2.14 that if we choose a Levi decomposition $\Pbf=\Wbf\rtimes\Gbf$, then $\Gbf^\der$ acts on $\Ubf$ trivially, because $\Ubf$ is of Hodge type $(-1,-1)$. In this case we still have a $\Qbb$-group $\Pbf^\der=\Wbf\rtimes\Gbf^\der$ of type $\Hscr$ by the same arguments as in \ref{group-law} and \ref{harish-chandra};  given $\Gamma\subset\Pbf(\Qbb)^+$ a congruence subgroup, we may still define the lattice space $\Omega=\Gamma^\dag\bsh\Pbf^\der(\Rbb)^+$. But the $\rho$-rigidity is never satisfied due to the existence of $\Ubf$, and the association of a lattice subspace to a special subvariety is no longer injective because conjugation by elements in $\Ubf$ is not reflected at the level of lattice subspaces. This is the main reason that our notion of $\rho$-rigidity, hence the ergodic-theoretic arguments for the main theorem, does not admit a natural generalization to the case of general mixed Shimura varieties. 

\end{remark}

 
Now we show that the main theorem \ref{main-theorem} implies the main corollary \ref{main-corollary}.

\begin{proof}[Proof of the Corollary]

Assume on the contrary that $\Sigma'_\maxx(Z)$ is infinite for some closed subset $Z$. Then $\Sigma'_\maxx(Z)$ contains an infinite sequence $(S_n) $. Let $\nu_n$ be the canonical measure associated to $S_n$, which form an infinite sequence in $\Pscr'(S)$. The compactness of $\Pscr'(S)$ implies the existence of a convergent subsequence. We may assume for simplicity that $(\nu_n) $ converges to some $\nu\in\Pscr'(S)$, and we write $S_\nu:$ for $\Supp\nu$. The "support convergence" property  implies

(i) $S_n\subset S_\nu$ for $n$ large, and therefore $S_\nu\nsubseteq Z$ by the maximality of the $S_n$'s;

(ii) $S_\nu$ is the archimedean closure of $\bigcup_{n> 0}\Supp\nu_n$, in particular it is contained in $Z$; which contradicts (i).\end{proof}

In the next two sections we will prove the main theorem for $S=\Omega$ and $S=M$ respectively. For simplicity we will assume that the datum $(\Pbf,Y)$ is
$\Cbf$-special itself, namely $\Cbf$ equals the connected center of $\Gbf$, thanks to the following reduction lemma:

\begin{lemma}
  \label{finite-maximality} Let $(\Pbf,Y)$ be a Kuga datum, and $\pi:(\Pbf,Y)\ra(\Gbf,X)$ the canonical projection onto the pure base. Fix $\Cbf$ a $\Qbb$-torus of $\Gbf$. Then the set of maximal $\Cbf$-special Kuga subdata of $(\Pbf,Y)$ is finite.
\end{lemma}

\begin{proof}
  If $(\Gbf',X')$ is a maximal $\Cbf$-special pure subdatum of $(\Gbf,X)$, then $(\pi^\inv(\Gbf')=\Vbf\rtimes\Gbf',\pi_*^\inv(X')=\Vbf(\Rbb)\rtimes X')$ is a maximal $\Cbf$-special Kuga subdatum of $(\Pbf,Y)$, and it contains any $\Cbf$-special subdatum whose image under $\pi$ is contained in $(\Gbf',X')$. Thus we are reduced to the case where $(\Pbf,Y)=(\Gbf,X)$ is pure, which is already done in \cite{ullmo-yafaev}. More precisely, E.Ullmo and A.Yafaev showed that:

  (1) (cf.\cite{ullmo-yafaev} 3.6) The connected centralizer $\Zbf_\Gbf\Cbf$ decomposes into an almost direct product $\Zbf^\circ\Hbf\Hbf^c$ where $\Zbf^\circ$ is the connected center of $\Zbf_\Gbf^\Cbf$, $\Hbf$ the product of non-compact $\Qbb$-factors, and $\Hbf^c$ the product of compact ones. Then by putting $\Gbf''=\Cbf\Hbf$ and $X''=\Gbf''(\Rbb)X'$, we get a pure subdatum $(\Gbf'',X'')$, which is a maximal $\Cbf$-special one by construction of $\Gbf$.

  (2) (cf.\cite{ullmo-yafaev} 3.7) For a given reductive $\Qbb$-subgroup $\Gbf''$ of $\Gbf$, there are at most finitely many pure subdata with Mumford-Tate groups equal to $\Gbf''$.\end{proof}

\section{Proof of the main theorem for lattice spaces}
The formulation of the main theorem is inspired from the following theorem of S.Mozes and N.Shah, which we state in the case of the lattice space $\Omega=\Gamma^\dag\bsh\Pbf^\der(\Rbb)^+$, with notations as in \ref{main-theorem}:

\begin{theorem}
  \label{mozes-shah} Let $\Pscr_h(\Omega)$ be the set of canonical measures on $\Omega$ associated to lattice subspaces of the form $\wp_\Gamma(\Hbf(\Rbb)^+)$ in the sense of \ref{kuga-lattice-space}, with $\Hbf$ running through $\Qbb$-subgroup of $\Pbf^\der$ of type $\Hscr$. Then $\Pscr_h(\Omega)$ is a compact subset of $\Pscr(\Omega)$ the set of Borelian probability measures on $\Omega$, endowed with the weak topology. Moreover it admits the "support convergence" property, namely if $(\nu_n) $ is a sequence in $\Pscr_h(\Omega)$ that converges to some $\nu\in\Pscr_h(\Omega)$, then $\Supp\nu_n\subset\Supp\nu$ for $n$ large enough, and the union $\bigcup_{n> 0}\Supp\nu_n$ is dense in $\Supp\nu$ for the archimedean topology.
\end{theorem}

From \ref{harish-chandra}(2) we see that $\Pscr'(\Omega)\subset\Pscr_h(\Omega)$. To prove the main theorem for $\Omega$ it remains to show that $\Pscr'(\Omega)$ is closed in $\Pscr_h(\Omega)$. We thus work within the following setting in the rest of this section:

\begin{assumption}
  We are given a sequence $(\nu_n) $ in $\Pscr'(\Omega)$ that converges to some $\nu$ in $\Pscr_h(\Omega)$. Let $\Omega_n$ be the support of $\nu_n$, which is defined by the $\rho$-rigid $\Cbf$-special subdatum $(\Pbf_n,Y_n)$, and $\Omega_\nu$ the support of $\nu$, defined by $\Pbf'$ a $\Qbb$-subgroup of type $\Hscr$ in $\Pbf^\der$. We have the decomposition $$(\Pbf_n,Y_n)=\Vbf_n\rtimes_\rho(v_n\Gbf_nv_n^\inv,v_n\rtimes X_n)$$ for $v_n\in\Vbf(\Qbb)$, such that the $\Qbb$-subgroups $\Gbf_n=\pi(\Pbf_n)\subset\Gbf$ have $\Cbb$ as the common connected center for all $n\in \Nbb$.

  For simplicity we assume that $\Supp\nu_n\subset\Supp\nu$ for all $n$. By computing the tangent spaces of $\Omega_n$ and of $\Omega'$ at the distinguished point $\Gamma^\dag e$, we get inclusions of Lie algebras $\Lie\Pbf^\der_{n\Rbb}\subset\Lie\Pbf'_\Rbb$, and thus $\Pbf_n^\der\subset\Pbf'$ for all $n$.
\end{assumption}


\begin{lemma}\label{push-forward}$\pi_*\nu_n$ is the canonical measure associated to $\pi(\Pbf_n^\der)$, and $\nu$ is the canonical measure associated to $\pi(\Pbf'^\der)$.
\end{lemma}

\begin{proof}
This is just a combination of the two constructions \ref{measure-compatibility}.
\end{proof}

The lemma above and the results in \cite{ullmo-yafaev} lead to the following:

\begin{lemma}
  The group $\pi(\Pbf')=\Hbf'$ is a semi-simple $\Qbb$-subgroup in $\Gbf^\der$ of type $\Hscr$, and is generated by $\bigcup_n\Gbf^\der_n$. It is centralized by $\Cbf$, and is $\rho$-rigid.  By putting $\Gbf'=\Cbf\Hbf'$ and $X'=\Gbf'(\Rbb)X_n$, we get a pure subdatum $(\Gbf',X')$, which is $\rho$-rigid and $\Cbf$-special. Finally, there are only finitely many $(\Gbf',X')$ obtained in this way when $n$ varies.
\end{lemma}

We then show that similar situation occurs for $\Pbf'$:

\begin{proposition}
 (1) The $\Qbb$-group $\Pbf'$ is generated by $\bigcup_n\Pbf^\der_n$.

 (2) For each $n$, the $\Qbb$-subgroup $\Pbf_n$  normalizes $\Pbf'$ in $\Pbf$, and the product $\Pbf_\nu=\Pbf'\Pbf_n$ is independent of $n$.

 (3) Let $Y_\nu$ be the orbit $\Pbf_\nu(\Rbb)Y_n$, then we get a $\rho$-rigid $\Cbf$-special subdatum $(\Pbf_\nu,Y_\nu)$. Only finitely many Kuga subdata arise in this way when $n$ runs through $\Nbb$.
\end{proposition}

\begin{proof}

(1) Note that $\Omega_n\subset\Omega'$ is a smooth submanifold passing through $\Gamma^\dag e$, with inclusion of tangent spaces $\Lie\Pbf^\der_{n\Rbb}\subset\Lie\Pbf'_\Rbb$, we see that $\Pbf'$ contains all the $\Pbf_n^\der$'s. Let $\Pbf''$ be the $\Qbb$-subgroup generated by $\bigcup_n\Pbf^\der_n$, then  the image $\pi(\Pbf'')$ is generated by $\bigcup_n\pi(\Pbf_n^\der)$, and is equal to $\pi(\Pbf')$ by the lemma above. In particular it is semi-simple without compact $\Qbb$-factors, and is centralized by $\Cbf$. The kernel $\Vbf''$ of $\Pbf''\ra\pi(\Pbf'')$ is a vectorial $\Qbb$-subgroup of $\Vbf$, and $\Pbf''$ is a $\Qbb$-subgroup in $\Pbf'$ of type $\Hscr$. The supports $\Omega_n$'s are contained in $\Omega'':=\wp_\Gamma(\Pbf''(\Rbb)^+)$, and by density we must have $\Omega''=\Omega'$, hence $\Pbf''=\Pbf'$.

(2) Consider the decompotion $\Pbf'=\Vbf'\rtimes(u\Hbf'u^\inv)$ for some $u\in\Vbf(\Qbb)$ and $\Hbf'=\pi(\Pbf')$. Then we have $\Pbf'\supset\Pbf_n^\der=\Vbf_n\rtimes(v_n\Hbf_nv_n^\inv)$ with $\Hbf_n=\Gbf_n^\der$. By the same arguments as used in \ref{injectivity}, we see that $u-v_n\in\Vbf'(\Qbb)$ for all $n$. Note that $\Cbf$ commutes with $\Hbf'$, and that $\Hbf'$ stabilizes $\Vbf'$, we see that $\Vbf'$ is stabilized by $u(\Cbf\Hbf')u^\inv$. Thus we have the inclusion chain  $$\Pbf'u\Cbf u^\inv=\Vbf'\rtimes(u\Cbf\Hbf'u^\inv)\supset\Vbf'\rtimes(v_n\Cbf\Hbf'v_n^\inv)\supset
\Vbf_n\rtimes(v_n\Gbf_nv_n^\inv)$$ hence the claim.

(3) Take $y\in Y_n$, it is clear that $(\Pbf_\nu,\Pbf_\nu(\Rbb)y)$ is a Kuga subdatum: it suffices to check the conditions on the Hodge type, which is true because $y(\Sbb)\subset\Pbf_{n\Rbb}\subset\Pbf_{\nu\Rbb}$ thus $\Lie\Pbf_\nu$ is a rational Hodge substructure of $\Lie\Pbf$ through $Ad_\Pbf\circ y$.\end{proof}

Thus $\nu=\lim_n\nu_n$ is associated to some $\rho$-rigid $\Cbf$-special subdatum, and the main theorem is established for the lattice space $\Omega$.

\section{Proof of the main theorem for Kuga varieties}
We first show that $\Pscr'(M)$ is compact, namely any sequence in $\Pscr'(M)$ admits a convergent subsequence with limit in $\Pscr'(M)$.

\begin{assumption}
 We keep the notations $\Omega=\Gamma^\dag\bsh\Pbf^\der(\Rbb)^+$, $M=\Gamma\bsh Y^+$, etc. as in \ref{main-theorem}. Let $(\mu_n)$ be a sequence in $\Pscr'(M)$, defined by $\rho$-rigid $\Cbf$-special connected Kuga subdata $(\Pbf_n,Y_n;Y_n^+)$. Denote by $\nu_n$ the canonical measure associated to $\Omega_n:=\wp_\Gamma(\Pbf_n^\der(\Rbb)^+)$ which is the lattice subspace given by $(\Pbf_n,Y_n)$, and $\mu_n=\kappa_{y_n *}\nu_n$ for some $y_n\in Y^+_n$.

  Using the main theorem for $\Omega$, we assume for simplicity that $(\nu_n)$ converges to some $\nu\in\Pscr'(\Omega)$, and that $\Supp\nu_n\subset\Supp\nu$ for all $n$; furthermore,  there is a connected $\rho$-rigid $\Cbf$-special subdatum $(\Pbf_\nu,Y_\nu;Y_\nu^+)$ which induces $\nu$ and contains infinitely many $(\Pbf_n,Y_n;Y^+_n)$, and we  may assume that $(\Pbf_n,Y_n;Y_n^+)\subset(\Pbf_\nu,Y_\nu;Y^+_\nu)$ for all $n$.
\end{assumption}

We  follow the strategy of \cite{clozel-ullmo} 4.3, 4.4, 4.5:

\begin{lemma}
Let $M=\Gamma\bsh Y^+$ be a Kuga variety defined by $(\Pbf,Y;Y^+,\Gamma)$, with pure base $S=\Gamma_\Gbf\bsh X^+$ given by $(\Gbf,X;X^+,\Gamma_\Gbf)$. Denote by $\Cbf$ the connected center of $\Gbf$. Then there exists a compact subset $K(\Cbf,M)$ of $Y^+$, such that if $M'\subset M$ is a $\Cbf$-special subvariety, then $M'=\wp_\Gamma(Y'^+)$  can be defined by some $\Cbf$-special subdatum $(\Pbf',Y';Y'^+)$ with $Y'^+\cap K(\Cbf,Y^+)\neq\emptyset$.

\end{lemma}

\begin{proof}
  The construction is known in the pure case $M=S$ by \cite{clozel-ullmo}. Since the condition on $\Cbf$-special Kuga subdata  only depends on the images under $\pi$, we may simply lift $K(\Cbf,S)$ into a compact subset of $Y^+$ as follows: choose a compact subset $C_\Vbf$ of $\Vbf(\Rbb)$ containing a fundamental domain for the action of $\Gamma_\Vbf(=\Ker(\Gamma\ra\Gamma_\Gbf))$ on $\Vbf(\Rbb)$, and take $K(\Cbf,M):=C_\Vbf\times X^+\cap\pi_*^\inv(K(\Cbf,S))$.
\end{proof}

Apply the lemma to the subvariety $M_\nu:=\wp_\Gamma(Y^+_\nu)\isom\Gamma_\nu\bsh Y^+_\nu$, we get a compact subset $K(\Cbf,M_\nu)$, and we may assume that the subdata $(\Pbf_n,Y_n;Y_n^+)$ of $(\Pbf_\nu,Y_\nu;Y_\nu^+)$ are chosen such that $Y_n^+\cap K(\Cbf,M_\nu)\neq\emptyset$. We thus choose $y_n\in Y_n^+\cap K(\Cbf, M_\nu)$, and assume further that $(y_n)$ converges to some $y\in K(\Cbf,M_\nu)\subset Y^+_\nu$ (for the archimedean topology). In this case we have $\mu_n=\kappa_{y_n*}\nu_n$, and $\mu:=\kappa_{y*}\nu$ is the canonical measure associated to $M_\nu$.

\begin{proposition}\label{compactness}
  The sequence $(\mu_n)$ converges to $\mu$ for the weak topology, and thus $\Pscr'(M)$ is compact.
\end{proposition}

\begin{proof}
It suffices to notice, analogous to the pure case in \cite{clozel-ullmo}, that $\kappa_{y_n}$ converge uniformly to $\kappa_{y}$ on each compact subset of $\Omega$, and thus for $f\in C_c(M)$, $f\circ\kappa_{y_n}$ converges to $f\circ\kappa_y$ uniformly on each compact subset of $M$. Since $\nu_n$ converges to $\nu$ for the weak topology, we get
$$\int_M f\kappa_{y*}\nu=\int_\Omega f\circ \kappa_y=\lim_n\int_\Omega f\circ\kappa_{y_n}\nu_n=\lim_n\int_Mf\kappa_{y_n*}\nu_n$$ for any $f\in C_c(M)$, hence the weak convergence $\kappa_{y*}\nu=\lim_n\kappa_{y_n*}\nu_n$.
\end{proof}

It remains to show that $\Pscr'(M)$ admits the "support convergence" property.

\begin{proposition}\label{support-convergence}
  Let $(\mu_n)$ be a sequence in $\Pscr'(M)$ that converges to some $\mu$. Write $M_n=\Supp\mu_n$ and $M_\mu=\Supp\mu$. Then $M_n\subset M_\mu$ for $n$ large enough, and $\bigcup_{n> 0}M_n$ is dense in $M_\mu$ for the archimedean topology.
\end{proposition}

\begin{proof}
  Assume on the contrary that there exists an infinite subsequence $(\mu_{n_m})$ such that $M_{n_m}\nsubseteq M_\mu$. We may simply assume that $(\mu_{n_m})$ is just $(\mu_n)$ itself.

 Write $\mu_n=\kappa_{y_n*}\nu_n$, with $\nu_n$ the canonical measure of the lattice subspace associated to $M_n$, given by the subdatum $(\Pbf_n,Y_n;Y_n^+)$, and $y_n\in Y_n^+$. Using the results in \ref{compactness}, we may, up to restricting to a subsequence, that $(\nu_n)$ converges to some $\nu\in\Pscr'(\Omega)$ given by some connected $\rho$-rigid $\Cbf$-special subdatum $(\Pbf',Y';Y'^+)$, that $(\Pbf_n,Y_n;Y_n^+)\subset(\Pbf',Y';Y'^+)$, and that $y_n$ converges to some $y\in Y'^+$. Then \ref{compactness} implies that $\mu_n=\kappa_{y_n*}\nu_n$ converges to $\kappa_{y*}\nu$, and it is clear that $M_n$ are all contained in the special subvariety $M':=\wp_\Gamma(Y'^+)$. We thus have $\mu=\kappa_{y*}\nu$ and $M'=M_\mu$, which contradicts the assumption that $M_n\nsubseteq M_\mu$ for all $n$.
\end{proof}

\end{document}